\documentclass[12pt]{article}

\usepackage[dvips]{graphicx}

\usepackage[english]{babel}
\usepackage{amssymb}
\usepackage{amsmath}

\newcommand{\bb}{\mathbb}
\newcommand{\conv}{\mathrm{conv}}

\newcommand{\rec}{\mathrm{rec}}
\newcommand{\et}{\mathrm{ext}}
\newcommand{\ep}{\mathrm{exp}}
\newcommand{\R}{\bb R}

\newcommand{\Z}{\bb Z}
\newcommand{\N}{\bb N}

\newcommand{\intr}{\mathrm{\bf int}}

\newcommand{\bd}{\mathrm{\bf bd}}
\newcommand{\ol}{\overline}

\newenvironment{proof}{\noindent\emph{Proof.}}{\hspace*{\stretch{1}}$\square$}

\newtheorem{prop}{Proposition}
\newtheorem{theorem}[prop]{Theorem}
\newtheorem{lemma}[prop]{Lemma}

\newtheorem{cor}[prop]{Corollary}

\newcommand{\sm}{\setminus}

\def\st{\,|\,}

\begin{document}
\title{Convex Sets and Minimal Sublinear Functions}

\author{Amitabh Basu\thanks{Carnegie Mellon University,
abasu1@andrew.cmu.edu}\and
G\'erard Cornu\'ejols \thanks{Carnegie Mellon University, gc0v@andrew.cmu.edu.
Supported by   NSF grant CMMI0653419,
ONR grant N00014-03-1-0188 and ANR grant BLAN06-1-138894.} \and
Giacomo Zambelli\thanks{
Universit\`a di Padova,
giacomo@math.unipd.it
}}
\date{April 2010}
\maketitle
\begin{abstract}
We show that, given a closed convex set $K$ containing the origin in its interior, the support function of the set $\{y\in K^*\st \exists x\in K\mbox{ such that } \langle x,y \rangle =1\}$ is the pointwise smallest among all sublinear functions $\sigma$ such that $K=\{x\st \sigma(x)\leq 1\}$.
\end{abstract}

\section{Introduction}

The purpose of this note is to prove the following theorem. For $K \subseteq \mathbb{R}^n$, we use the notation
\begin{eqnarray*}
K^* & = &\{y\in\R^n\st \langle x,y \rangle \leq 1 \mbox{ for all } x\in K\} \\
\hat K & = &\{y\in K^* \st \langle x,y \rangle = 1 \mbox{ for some } x\in K\}.
\end{eqnarray*}
The set $K^*$ is the {\em polar} of $K$. The set $\hat K$ is contained in the relative boundary of $K^*$.
The polar $K^*$ is a convex set whereas $\hat K$ is not convex in general.

The {\em support function} of a nonempty set $T\subset \R^n$ is defined by
$$\sigma_T(x)=\sup_{y\in T}\langle x,y \rangle \quad \mbox{for all } x\in\R^n.$$
It is straightforward to show that support functions are {\em sublinear}, that is they are convex and positively homogeneous (A function $f:\,\R^n\rightarrow \R$ is {\em positively homogeneous} if $f(t x)=tf(x)$ for every $x\in\R^n$ and $t>0$), and $\sigma_T = \sigma_{\overline{\conv} (T)}$~\cite{Hir-Lem}.
 We will show that, if $K\subset \R^n$ is a closed convex set containing the origin in its interior, then  $K=\{x\st \sigma_{\hat K}(x)\leq 1\}$. The next theorem shows that $\sigma_{\hat K}$ is the smallest sublinear function with this property.

\begin{theorem}\label{thm:sublinear}
Let $K\subset \R^n$ be a closed convex set containing the origin in its interior. If $\sigma:\,\R^n\rightarrow \R$ is a sublinear function such that $K=\{x \in \R^n \st \sigma(x)\leq 1\}$, then $\sigma_{\hat K}(x)\leq \sigma(x)$ for all $x\in\R^n$.
\end{theorem}

In the remainder we define $\rho_K = \sigma_{\hat K}$.

\medskip
Let $K\subset \R^n$ be a closed convex set containing the origin in its interior. A standard concept in convex analysis~\cite{Hir-Lem,Rock} is that of {\em gauge} (sometimes called Minkowski function), which is the function $\gamma_K$  defined by $$\gamma_K(x)=\inf\{t>0\st  t^{-1}x\in K\} \quad \mbox{ for all } x\in \R^n.$$
By definition $\gamma_K$ is nonnegative. One can readily verify that $K=\{x\st\gamma_K(x)\leq 1\}$. It is well known that $\gamma_K$ is the support function of $K^*$ (see~\cite{Hir-Lem} Proposition 3.2.4).

Given any sublinear function $\sigma$ such that $K=\{x\st \sigma(x)\leq 1\}$, it follows from positive homogeneity that $\sigma(x)=\gamma_K(x)$ for every $x$ where $\sigma(x)>0$. Hence $\sigma(x)\leq \gamma_K(x)$ for all $x\in\R^n$. On the other hand, we prove in Theorem~\ref{thm:sublinear} that the sublinear function $\rho_K$ satisfies $\rho_K(x)\leq\sigma(x)$ for all $x\in\R^n$. In words, $\gamma_K$ is the largest sublinear function such that $K=\{x \in \R^n \st \sigma(x)\leq 1\}$ and $\rho_K$ is the smallest.
\medskip

Note that $\rho_K$ can take negative values, so in general it is different from the gauge $\gamma_K$. Indeed the recession cone of $K$, which is the set $\rec(K)=\{x\in K\st t x\in K \mbox{ for all } t>0\}$, coincides with $\{x\in K\st \sigma(x)\leq 0\}$ for every sublinear function $\sigma$ such that $K=\{x\st \sigma(x)\leq 1\}$. In particular $\rho_K(x)$ can be negative for $x\in\rec(K)$.
For example, let $K=\{x\in\R^2\st x_1\leq 1,\, x_2\leq 1\}$. Then $K^*=\conv\{(0,0),(1,0),(0,1)\}$ and $\hat K=\conv\{(1,0),(0,1)\}$. Therefore, for every $x\in\R^2$, $\gamma_K(x)=\max\{0,x_1,x_2\}$ and $\rho_K(x)=\max\{x_1,x_2\}$. In particular, $\rho_K(x)<0$ for every $x$ such that $x_1<0$, $x_2<0$.

\medskip
By H\"ormander's theorem \cite{hor}, a sublinear function $\sigma:\,\R^n\rightarrow \R$ is the support function of a unique bounded closed convex set $C \subset \R^n$, say $\sigma = \sigma_C$. So the condition $K=\{x \in \R^n \st \sigma(x)\leq 1\}$ says $K = C^*$. Thus Theorem~\ref{thm:sublinear} can be restated in its set version.

\begin{theorem}\label{thm:set-version}
Let $K\subset \R^n$ be a closed convex set containing the origin in its interior. If $C\subset \R^n$ is a bounded closed convex set such that $K=C^*$, then $\hat K \subset C$.
\end{theorem}

When $K$ is bounded, this theorem is trivial (the hypothesis $K=C^*$ becomes $K^*=C$. The conclusion $\hat K \subset C$ is obvious because one always has $\hat K \subset K^*$.) So the interesting case of Theorem~\ref{thm:sublinear} is when $K$ is unbounded.

\medskip
We present the proof of Theorem~\ref{thm:sublinear} in Section~\ref{SEC:proof}. Theorem~\ref{thm:sublinear} has applications in integer programming. In particular it is used to establish the relationship between minimal inequalities and maximal lattice-free convex sets \cite{basu-conf-cor-zam}, \cite{bas-conf-cor-zam}. We summarize these results in Section~\ref{SEC:IP}.

\section{Proof of Theorem~\ref{thm:sublinear}}
\label{SEC:proof}

We will need Straszewicz's theorem~\cite{Str} (see~\cite{Rock} Theorem 18.6). Given a closed convex set $C$, a point $x\in C$ is {\em extreme} if it cannot be written as a proper convex combination of two distinct points in $C$. A point $x\in C$ is {\em exposed} if there exists a supporting hyperplane $H$ for $C$ such that $H\cap C=\{x\}$. Clearly exposed points are extreme. We will denote by $\et(C)$ the set of extreme points and $\ep(C)$ the set of exposed points of $C$.

\begin{theorem}\label{thm:str}
Given a closed convex set $C$, the set of exposed points of $C$ is a dense subset of the set of extreme points of $C$.
\end{theorem}

Let $K$ be a closed convex set containing the origin in its interior. Let $\sigma$ be a sublinear function such that $K=\{x\st \sigma(x)\leq 1\}$. The boundary of $K$, denoted by $\bd(K)$, is the set $\{x\in K\st \sigma(x)=1\}$.

\begin{lemma}~\label{lemma:not-rec} For every $x\notin\rec(K)$, $\sigma(x)=\rho_K(x)=\sup_{y\in K^*} \langle x,y \rangle$.
\end{lemma}
\begin{proof} Let $x\notin\rec(K)$. Then $t=\sigma(x)>0$. By positive homogeneity, $\sigma(t^{-1}x)=1$, hence $t^{-1}x\in\bd(K)$.
Since $K$ is closed and convex, there exists a supporting hyperplane for $K$ containing $t^{-1}x$. Since $0\in\intr(K)$, this implies that there exists $\bar y\in K^*$ such that $(t^{-1}x)\bar y=1$. In particular $\bar y\in \hat K$, hence by definition $\rho_K(x)\geq \langle x, \bar y \rangle   = t$.

Furthermore, for any $y\in K^*$, $\langle t^{-1}x,y \rangle \leq 1$, hence $\langle x,y \rangle \leq t$, which implies $t\geq \sup_{y\in K^*} \langle x,y \rangle$. Thus $$\rho_K(x)\geq t\geq \sup_{y\in K^*} \langle x,y \rangle \geq \sup_{y\in \hat K} \langle x,y \rangle =\rho_K(x),$$
where the last inequality holds since $\hat K\subset K^*$, hence equality holds throughout.
\end{proof}

\begin{cor}
$K=\{x\st \rho_K(x)\leq 1\}$.
\end{cor}

\begin{lemma}\label{lemma:ep-0}Given an exposed point $\bar y$ of $K^*$ different from the origin, there exists $x\in K$ such that $\langle x,\bar y \rangle =1$ and $\langle x,y \rangle <1$ for all $y\in K^*$ distinct from $\bar y$.
\end{lemma}
\begin{proof} If $\bar y\neq 0$ is an exposed point of $K^*$, then there exists a supporting hyperplane $H=\{y\st \langle a, y \rangle  = \beta\}$ such that $\langle a, \bar y \rangle =\beta$ and $\langle a, y \rangle < \beta$ for every $y\in K^*\sm\{\bar y\}$. Since $0\in K^*$ and $\bar y\neq 0$, $\beta >0$. Thus the point $x=\beta^{-1}a\in K^{**}=K$ satisfies the statement (where $K=K^{**}$ holds because $0\in K$).
\end{proof}

\bigskip
The next lemma states that $\hat K$ and $\hat K \cap \ep(K^*)$ have the same support function.

\begin{lemma}\label{lemma:just-ep} For every $x\in \R^n$, $\rho_K(x)=\sup_{y\in \hat K\cap\ep(K^*)} \langle x,y \rangle $.
\end{lemma}
\begin{proof}
We first show that $\rho_K(x)=\sup_{y\in \hat K\cap\et(K^*)} \langle x,y \rangle$. Given $y\in \hat K$ we show that there exists an extreme point $y'$ of $K^*$ in $\hat K$ such that $\langle x,y \rangle \leq \langle x,y' \rangle$. Since $y\in \hat K$, there exists $\bar x\in K$ such that $ \langle \bar x,y \rangle =1$.  The point $y$ is a convex combination of extreme points $y_1,\ldots, y_k$ of $K^*$, and each $y_i$ satisfies $\langle \bar x,y^i \rangle =1$. Thus $y^1,\ldots,y^k\in \hat K$, and $\langle x,y^i \rangle \geq \langle x,y \rangle$ for at least one $i$.

By Straszewicz's theorem (Theorem~\ref{thm:str}) the set of exposed points in $K^*$ is a dense subset of the extreme points of $K^*$. By Lemma~\ref{lemma:ep-0}, all exposed points of $K^*$ except the origin are in $\hat K$, hence $\ep(K^*)\cap \hat K$ is dense in $\et(K^*)\cap \hat K$. Therefore $\rho_K(x)=\sup_{y\in \hat K\cap\ep(K^*)} \langle x,y \rangle$.
\end{proof}

\bigskip

A function $\sigma$ is {\em subadditive} if $\sigma(x_1+x_2)\leq \sigma(x_1)+\sigma(x_2)$ for every $x_1,x_2\in\R^n$. It is easy to show that $\sigma$ is sublinear if and only if it is subadditive and positively homogeneous.
\bigskip
\bigskip

\noindent{\em Proof of Theorem~\ref{thm:sublinear}.} By Lemma~\ref{lemma:not-rec}, we only need to show $\sigma(x)\geq \rho_K(x)$ for points $x\in\rec(K)$. By Lemma~\ref{lemma:just-ep} it is sufficient to show that, for every exposed point $\bar y$ of $K^*$ contained in $\hat K$, $\sigma(x)\geq \langle x,\bar y \rangle$.\medskip

Let $\bar y$ be an exposed point of $K^*$ in $\hat K$. By Lemma~\ref{lemma:ep-0} there exists $\bar x\in K$ such that $\langle \bar x,\bar y \rangle =1$ and $\langle \bar x, y \rangle <1 $ for all $y\in K^*$ distinct from $\bar y$. Note that $\bar x\in\bd(K)$.

We observe that for all $\delta>0$, $\bar x-\delta^{-1} x\notin\rec(K)$.  Indeed, since $x\in\rec(K)$, $\bar x+\delta^{-1} x\in K$. Hence $\bar x-\delta^{-1} x\notin\intr(K)$ because $\bar x\in\bd(K)$. Since $0\in\intr(K)$ and $\bar x-\delta^{-1} x\notin\intr(K)$, then $\bar x-\delta^{-1} x\notin\rec(K)$. Thus by Lemma~\ref{lemma:not-rec}
\begin{equation}\label{eq:delta}
\sigma(\bar x-\delta^{-1} x)=\sup_{y\in K^*} \langle \bar x-\delta^{-1} x,y \rangle .
\end{equation}

Since $\bar x\in\bd(K)$, $\sigma(\bar x)=1$. By subadditivity, $1=\sigma (\bar x)\leq \sigma(\bar x-\delta^{-1} x)+\sigma(\delta^{-1} x).$
By positive homogeneity, the latter implies that $\sigma(x)\geq \delta-\delta \sigma(\bar x-\delta^{-1} x)$
for all $\delta>0$. By~(\ref{eq:delta}),
$$\sigma(x)\geq \inf_{y\in K^*} [\delta (1-\langle \bar x,y \rangle) + \langle x,y \rangle ],\quad\forall \delta>0$$
hence
$$\sigma(x)\geq \sup_{\delta>0}\inf_{y\in K^*} [\delta (1-\langle \bar x,y \rangle)  +\langle x,y \rangle ].$$
Let $g(\delta)=\inf_{y\in K^*} \delta (1-\langle \bar x,y \rangle) + \langle x,y \rangle$. Since $\bar x\in K$, $1-\langle \bar x, y \rangle \geq 0$ for every $y\in K^*$. Hence $\delta (1-\langle \bar x,y \rangle) +\langle x,y \rangle$ defines an increasing affine function of $\delta $ for each $y\in K^*$, therefore $g(\delta)$ is increasing and concave. Thus $\sup_{\delta>0}g(\delta)=\lim_{\delta\rightarrow \infty}g(\delta)$.

Since $0\in\intr(K)$, $K^*$ is compact. Hence, for every $\delta>0$ there exists $y(\delta)\in K^*$ such that $g(\delta)=\delta (1-\langle \bar x,y(\delta) \rangle)+\langle x,y(\delta) \rangle$. Furthermore, there exists a sequence $(\delta_i)_{i\in\N}$ such that $\lim_{i\rightarrow \infty}\delta_i=+\infty$ and the sequence $(y_i)_{i\in\N}$ defined by $y_i=y(\delta_i)$ converges, because in a compact set every sequence has a convergent subsequence. Let $y^*=\lim_{i\rightarrow \infty}y_i$.
\medskip

We conclude the proof by showing that $\sigma(x)\geq \langle x,y^* \rangle$ and $y^*=\bar y$.
\begin{eqnarray*}
\sigma(x)\geq \sup_{\delta>0} g(\delta) &=& \lim_{i\rightarrow \infty} g(\delta_i)\\
&=&\lim_{i\rightarrow \infty} [\delta_i (1-\langle \bar x,y_i \rangle)+\langle x,y_i \rangle]\\
&=&\lim_{i\rightarrow \infty} \delta_i (1-\langle \bar x,y_i \rangle )+\langle x,y^*\rangle \\
&\geq & \langle x,y^* \rangle
\end{eqnarray*}
where the last inequality follows from the fact that $\delta_i (1-\langle \bar x,y_i\rangle )\geq 0$ for all $i\in\N$.
Finally, since $\lim_{i\rightarrow \infty} \delta_i (1-\langle \bar x,y_i \rangle )$ is bounded and $\lim_{i\rightarrow \infty} \delta_i=+\infty$, it follows that
$\lim_{i\rightarrow \infty}(1-\langle \bar x,y_i \rangle )=0$, hence $\langle \bar x,y^* \rangle =1$. By our choice of $\bar x$, $\langle \bar x, y \rangle  <1$ for every $y\in K^*$ distinct from $\bar y$. Hence $y^*=\bar y$. \hfill $\Box$

\section{An application to integer programming}
\label{SEC:IP}

In~\cite{basu-conf-cor-zam} and \cite{bas-conf-cor-zam}, Basu et al.
apply Theorem~\ref{thm:sublinear} to cutting plane theory.

Consider a mixed integer linear program, and the optimal tableau of the linear programming
relaxation. We select $n$ rows of the tableau, relative to $n$ basic integer variables
$x_1,\ldots, x_n$.  Let $s_1,\ldots, s_m$ denote the nonbasic variables. Let $f_i$ be the value of $x_i$ in the basic solution associated with the tableau, $i=1,\ldots,n$, and suppose $f\notin\Z^n$. The tableau
restricted to these $n$ rows is of the form
\begin{equation}\label{eq:tableau}x=f+\sum_{j=1}^m  r^js_j,\quad x\in P\cap\Z^n,\, s\geq 0,\, \mbox { and } s_j\in \Z, j\in I,\end{equation}
where $r^j\in\R^n$, $j=1,\ldots, m$, $I$ denotes the set of
integer nonbasic variables, and $P$ is some full-dimensional rational polyhedron in $\R^n$, representing constraints on the basic variables (typically nonnegativity or bounds on the variables).

An important question in integer programming is to derive valid inequalities for~\eqref{eq:tableau}, cutting off the current infeasible solution $x=f$, $s=0$. We consider a simplified model where the integrality conditions are relaxed on all nonbasic variables.
So we study the following model, introduced by Johnson~\cite{Johnson},
\begin{equation}\label{eq:finite}x=f+\sum_{j=1}^m r^js_j,\quad x\in S,\, s\geq 0,\end{equation}
where where $S = P\cap\Z^n$ and $f\in\conv(S)\setminus\Z^n$.
Note that every inequality cutting off the point $(f,0)$ can be expressed in terms of the nonbasic variables $s$ only, and can therefore be written in the form $\sum_{j=1}^m \alpha_j s_j\geq 1$.

Basu et al.~\cite{bas-conf-cor-zam} study ``general formulas'' to generate such inequalities. By this, we mean functions $\psi\,:\,\R^n\rightarrow \R$ such that the inequality
$$\sum_{j=1}^m \psi(r^j) s_j\geq 1$$
is valid for~\eqref{eq:finite} for every choice of $m$ and vectors $r^1,\ldots,r^m\in\R^n$. We refer to such functions $\psi$ as {\em valid functions} (with respect to $f$ and $S$). Since one is interested in the deepest inequalities cutting off $(f,0)$, one only needs to investigate (pointwise) {\em minimal valid functions}.

Given a sublinear function $\psi$ such that the set \begin{equation}\label{eq:B_psi} B_\psi=\{x\in\R^n\st \psi(x-f)\leq 1\}\end{equation}
is {\em $S$-free} (i.e. $\intr(B_{\psi})\cap S=\emptyset$), it is easily shown that $\psi$ is a valid function. Indeed, since $\psi$ is sublinear, $B_\psi$ is a closed convex set with $f$ in its interior, thus,   given any solution $(\bar x,\bar s)$ to~\eqref{eq:finite}, we have
$\sum_{j=1}^m \psi(r^j)\bar s_j\geq \psi(\sum_{j=1}^m r^j\bar s_j)=\psi(\bar x-f)\geq 1,$
where the first inequality follows from sublinearity and the last one from the fact that $\bar x\notin\intr(B_\psi)$.

On the other hand, Dey and Wolsey~\cite{DW} show that, if $\psi$ is a minimal valid function, $\psi$ is sublinear and $B_\psi$ is an $S$-free convex set with $f$ in its interior.
\smallskip

Using Theorem~\ref{thm:sublinear}, Basu et al. \cite{bas-conf-cor-zam} prove that, if $\psi$ is a minimal valid function, then $B_\psi$ is a {\em maximal $S$-free convex set}. That is, $B_\psi$ is an inclusionwise maximal convex set such that $\intr(B_\psi)\cap S=\emptyset$. Furthermore, they give a explicit formulas for all minimal valid functions.

In order to prove this result, they first show that maximal $S$-free convex sets are polyhedra.
Therefore, a maximal $S$-free convex set $B\subseteq \R^n$ containing $f$ in its interior can be uniquely written in the form $B=\{x\in \R^n:\;\langle a_i, x-f \rangle \le 1,\;  i=1,\ldots,k\}$. Thus, if we let $K := \{x-f\st x\in B\}$, Theorem~\ref{thm:sublinear} implies that the function $\psi_B:=\rho_K$ is the minimal sublinear function such that $B=\{x\in\R^n\st \psi_B(x-f)\leq 1\}$. Note that, since $K^*=\conv\{0,a_1,\ldots,a_k\}$,  $\psi_B$ has the following simple form
\begin{equation}\label{eq:psiB}\psi_B(r)=\max_{i=1,\ldots,k}\langle a_i,r\rangle,\quad \forall r\in\R^n.\end{equation}

From the above, it is immediate that, if $B$ is a maximal $S$ free convex set, then the function $\psi_B$ is a minimal valid function.

The main use of Theorem~\ref{thm:sublinear} is in the proof of the converse statement, namely, that every minimal valid function is of the form $\psi_B$ for some maximal $S$-free convex set $B$ containing $f$ in its interior. The proof outline is as follows. Suppose $\psi$ is a minimal valid function. Thus $\psi$ is sublinear and $B_\psi$ is an $S$-free convex set. Let $K:= \{x-f\st x\in B_\psi\}$.
\smallskip

\noindent $i)$ Since  $\{x\in\R^n\st \rho_K(x-f)\leq 1\}=B_\psi$,  Theorem~\ref{thm:sublinear} implies that $\psi\geq\rho_K$.
\smallskip

\noindent $ii)$ Theorem~\cite{bas-conf-cor-zam}.  {\em There exists a maximal $S$-free convex set $B=\{x\in \R^n:\; \langle a_i, x-f \rangle \le 1,\;  i=1,\ldots,k\}$ such that $a_i \in \ol\conv (\hat K ) $ for $i = 1, \ldots , k$.}

\smallskip
\noindent  $iii)$ For every $r\in \R^n$, we have $\psi (r) \geq \rho_K(r) = \sup_{y \in \ol\conv (\hat K ) } \langle y, r \rangle \geq \max_{i = 1, \ldots , k} \langle a_i, r \rangle = \psi_B (r)$, where the first inequality follows from i) and the second  from ii). Since $\psi_B$ is a valid function, it follows by the minimality of $\psi$ that $\psi=\psi_B$.

\end{document}